\documentclass{amsart}
\usepackage{amssymb}

\title{On Combinatorial Formulas for Macdonald Polynomials}

\author{Cristian Lenart}

\address{Department of Mathematics and Statistics, State University of New York at Albany, Albany, NY 12222}
\email{lenart@albany.edu}



\date{April 30, 2008}

\thanks{Cristian Lenart was partially supported by the National Science Foundation grant  DMS-0701044}

\DeclareMathOperator{\Des}{Des}
\DeclareMathOperator{\Diff}{Diff}
\DeclareMathOperator{\maj}{maj}
\DeclareMathOperator{\inv}{inv}
\DeclareMathOperator{\rt}{r}
\DeclareMathOperator{\rev}{rev}
\DeclareMathOperator{\arm}{arm}
\DeclareMathOperator{\leg}{leg}

\newlength{\cellsize}
\newcommand\tableau[1]{
\vcenter{
\let\\=\cr
\baselineskip=-16000pt
\lineskiplimit=16000pt
\lineskip=0pt
\halign{&\tableaucell{##}\cr#1\crcr}}}

\newcommand{\tableaucell}[1]{{%
\def \arg{#1}\def \void{}%
\ifx \void \arg
\vbox to \cellsize{\vfil \hrule width \cellsize height 0pt}%
\else
\unitlength=\cellsize
\begin{picture}(1,1)
\put(0,0){\makebox(1,1){$#1$}}
\put(0,0){\line(1,0){1}}
\put(0,1){\line(1,0){1}}
\put(0,0){\line(0,1){1}}
\put(1,0){\line(0,1){1}}
\end{picture}%
\fi}}

\numberwithin{equation}{section}

\theoremstyle{plain}
\newtheorem{theorem}{Theorem}[section]
\newtheorem{proposition}[theorem]{Proposition}

\newtheorem{definition}[theorem]{Definition}

\newtheorem{example}[theorem]{Example}

\theoremstyle{remark}
\newtheorem{remark}[theorem]{Remark}

\def\R{\mathbb{R}}
\def\Z{\mathbb{Z}}

\def\F{{\mathcal F}}

\def\Waff{W_{\mathrm{aff}}}

\def\h{\mathfrak{h}}
\def\hR{\mathfrak{h}^*_\mathbb{R}}
\newcommand{\casethree}[6]{\left\{ \begin{array}{ll} #1 &\mbox{if $#2$} \\[0.04in]#3 &\mbox{if $#4$} \\ [0.04in] #5 & \mbox{if $#6$}\,. \end{array} \right.}

\newcommand{\casetwo}[3]{\left\{ \begin{array}{ll} #1 &\mbox{if $#2$} \\ [0.04in] #3 &\mbox{otherwise}\,. \end{array} \right.}
\newcommand{\case}[3]{\left\{ \begin{tabular}{ll}  #1 & if #2 \\ [.07in] #3 & otherwise\,. \end{tabular} \right.}

\newcommand{\stacksum}[2]{\sum_{\begin{array}{c}\vspace{-5.4mm}\;\\ \vspace{-1mm}\scriptstyle{#1}\\ \scriptstyle{#2}\end{array}} }

\begin{document}
\bibliographystyle{plain}

\begin{abstract} 
A recent breakthrough in the theory of (type $A$) Macdonald polynomials is due to Haglund, Haiman and Loehr, who exhibited a combinatorial formula for these polynomials in terms of a pair of statistics on fillings of Young diagrams. Ram and Yip gave a formula for the Macdonald polynomials of arbitrary type in terms of so-called alcove walks; these originate in the work of Gaussent-Littelmann and of the author with Postnikov on discrete counterparts to the Littelmann path model. In this paper, we relate the above developments, by explaining how the Ram-Yip formula compresses to a new formula, which is similar to the Haglund-Haiman-Loehr one but contains considerably fewer terms. 
\end{abstract}

\maketitle


\section{Introduction}
\label{intro}

Macdonald \cite{macsft,macopa} defined a remarkable family of orthogonal polynomials depending on parameters $q,t$, which bear his name. These polynomials generalize the spherical functions for a $p$-adic group, the Jack polynomials, and the zonal polynomials. At $q=0$, the Macdonald polynomials specialize to the Hall-Littlewood polynomials, and thus they further specialize to the Weyl characters (upon setting $t=0$ as well). There has been considerable interest recently in the combinatorics of Macdonald polynomials. This stems in part from a combinatorial formula for the ones corresponding to type $A$, which is due to Haglund, Haiman, and Loehr \cite{hhlcfm}, and which is in terms of fillings of Young diagrams. This formula uses two statistics on the mentioned fillings, called inv and maj. The Haglund-Haiman-Loehr formula already found important applications, such as new proofs of the positivity theorem for Macdonald polynomials, which states that the two-parameter Kostka-Foulkes polynomials have nonnegative integer coefficients. One of the mentioned proofs, due to Grojnowski and Haiman \cite{gahaha}, is based on Hecke algebras, while the other, due to Assaf \cite{asasem} is purely combinatorial and leads to a positive formula for the two-parameter Kostka-Foulkes polynomials.

Schwer \cite{schghl} gave a formula for the Hall-Littlewood polynomials of arbitrary type (cf. also \cite{ramawh}). This formula is in terms of so-called alcove walks, which originate in the work of Gaussent-Littelmann \cite{gallsg} and of the author with Postnikov \cite{lapcmc,lapawg} on discrete counterparts to the Littelmann path model \cite{litlrr,litpro}. Schwer's formula was recently generalized by Ram and Yip to a similar formula for the Macdonald polynomials \cite{raycfm}. The generalization consists in the fact that the latter formula is in terms of alcove walks with both ``positive'' and ``negative'' foldings, whereas in the former only ``positive'' foldings appear.

In \cite{lenhlp}, we relate Schwer's formula to the Haglund-Haiman-Loehr formula. More precisely, we show that we can group the terms in the type $A$ instance of Schwer's formula into equivalence classes, such that the sum in each equivalence class is a term in the Haglund-Haiman-Loehr formula for $q=0$.

In this paper, we relate the Ram-Yip formula to the Haglund-Haiman-Loehr formula. In a similar way to \cite{lenhlp}, we show that we can group the terms in the type $A$ instance of the Ram-Yip formula into equivalence classes, such that the sum in each class is a term in a new formula, which is similar to the Haglund-Haiman-Loehr one but contains considerably fewer terms. An equivalence class consists of all the terms corresponding to alcove walks that produce the same filling of a Young diagram $\lambda$ (indexing the Macdonald polynomial) via a simple construction. In fact, in this paper we require that the partition $\lambda$ is a regular weight; the general case will be considered in a future publication. 

This work does not directly specialize to the one in \cite{lenhlp} because here we do not recover the Haglund-Haiman-Loehr formula, but one similar to it. The explanation is that the Ram-Yip formula, which we use as input, is not a direct generalization of Schwer's formula. The main difference consists in the choice of a $\lambda$-chain (or reduced alcove path) in the two formulas, cf. Section \ref{specschwer} compared to \cite{lenhlp}[Section 3.1].

\medskip

{\bf Acknowledement.} I am grateful to Jim Haglund for helpful discussions.

\section{Preliminaries}\label{prelim}

We recall some background information on finite root systems and affine Weyl groups.

\subsection{Root systems}\label{rootsyst}

Let $\mathfrak{g}$ be a complex semisimple Lie algebra, and $\h$ a Cartan subalgebra, whose rank is $r$.
Let $\Phi\subset \h^*$ be the 
corresponding irreducible {\it root system}, $\hR\subset \h^*$ the real span of the roots, and $\Phi^+\subset \Phi$ the set of positive roots. Let $\rho:=\frac{1}{2}(\sum_{\alpha\in\Phi^+}\alpha)$. 
Let $\alpha_1,\ldots,\alpha_r\in\Phi^+$ be the corresponding 
{\it simple roots}.
We denote by $\langle\,\cdot\,,\,\cdot\,\rangle$ the nondegenerate scalar product on $\hR$ induced by
the Killing form.  
Given a root $\alpha$, we consider the corresponding {\it coroot\/} $\alpha^\vee := 2\alpha/\langle\alpha,\alpha\rangle$ and reflection $s_\alpha$.  

Let $W$ be the corresponding  {\it Weyl group\/}, whose Coxeter generators are denoted, as usual, by $s_i:=s_{\alpha_i}$. The length function on $W$ is denoted by $\ell(\,\cdot\,)$. The {\em Bruhat graph} on $W$ is the directed graph with edges $u\rightarrow w$ where $w=u s_{\beta}$ for some $\beta\in\Phi^+$, and $\ell(w)>\ell(u)$; we usually label such an edge by $\beta$ and write $u\stackrel{\beta}\longrightarrow w$. The {\em reverse Bruhat graph} is obtained by reversing the directed edges above. The {\em Bruhat order} on $W$ is the transitive closure of the relation corresponding to the Bruhat graph.

The {\it weight lattice\/} $\Lambda$ is given by
\begin{equation}
\Lambda:=\{\lambda\in \hR \::\: \langle\lambda,\alpha^\vee\rangle\in\Z
\textrm{ for any } \alpha\in\Phi\}.
\label{eq:weight-lattice}
\end{equation}
The weight lattice $\Lambda$ is generated by the 
{\it fundamental weights\/}
$\omega_1,\ldots,\omega_r$, which form the dual basis to the 
basis of simple coroots, i.e., $\langle\omega_i,\alpha_j^\vee\rangle=\delta_{ij}$.
The set $\Lambda^+$ of {\it dominant weights\/} is given by
$$
\Lambda^+:=\{\lambda\in\Lambda \::\: \langle\lambda,\alpha^\vee\rangle\geq 0
\textrm{ for any } \alpha\in\Phi^+\}.
$$
Let $\Z[\Lambda]$ be the group algebra of the weight lattice $\Lambda$, which  has
a $\Z$-basis of formal exponents $\{x^\lambda \::\: \lambda\in\Lambda\}$ with
multiplication $x^\lambda\cdot x^\mu := x^{\lambda+\mu}$.

Given  $\alpha\in\Phi$ and $k\in\Z$, we denote by $s_{\alpha,k}$ the reflection in the affine hyperplane
\begin{equation}
H_{\alpha,k} := \{\lambda\in \hR \::\: \langle\lambda,\alpha^\vee\rangle=k\}.
\label{eqhyp}
\end{equation}
These reflections generate the {\it affine Weyl group\/} $\Waff$ for the {\em dual root system} 
$\Phi^\vee:=\{\alpha^\vee \::\: \alpha\in\Phi\}$. 
The hyperplanes $H_{\alpha,k}$ divide the real vector space $\hR$ into open
regions, called {\it alcoves.} 
The {\it fundamental alcove\/} $A_\circ$ is given by 
$$
A_\circ :=\{\lambda\in \hR \::\: 0<\langle\lambda,\alpha^\vee\rangle<1 \textrm{ for all }
\alpha\in\Phi^+\}.
$$

\subsection{Alcove walks}\label{alcovewalks}

We say that two alcoves $A$ and $B$ are {\it adjacent} 
if they are distinct and have a common wall.  
Given a pair of adjacent alcoves $A\ne B$ (i.e., having a common wall), we write 
$A\stackrel{\beta}\longrightarrow B$ if the common wall 
is of the form $H_{\beta,k}$ and the root $\beta\in\Phi$ points 
in the direction from $A$ to $B$.  

\begin{definition}
An {\em alcove path\/} is a sequence of alcoves
 such that any two consecutive ones are adjacent. 
We say that an alcove path $(A_0,A_1,\ldots,A_m)$ is {\it reduced\/} if $m$ is the minimal 
length of all alcove paths from $A_0$ to $A_m$.
\end{definition}

We need the following generalization of alcove paths.

\begin{definition}\label{defalcwalk} An {\em alcove walk} is a sequence 
$\Omega=(A_0,F_1,A_1, F_2, \ldots , F_m, A_m, F_{\infty})$ 
such that $A_0,\ldots,$ $A_m$ are alcoves; 
$F_i$ is a codimension one common face of the alcoves $A_{i-1}$ and $A_i$,
for $i=1,\ldots,m$; and 
$F_{\infty}$ is a vertex of the last alcove $A_m$. The weight $F_\infty$ is called the {\em weight} of the alcove walk, and is denoted by $\mu(\Omega)$. 
\end{definition}
 
The {\em folding operator} $\phi_i$ is the operator which acts on an alcove walk by leaving its initial segment from $A_0$ to $A_{i-1}$ intact and by reflecting the remaining tail in the affine hyperplane containing the face $F_i$. In other words, we define
$$\phi_i(\Omega):=(A_0, F_1, A_1, \ldots, A_{i-1}, F_i'=F_i,  A_{i}', F_{i+1}', A_{i+1}', \ldots,  A_m', F_{\infty}')\,;$$
here $A_j' := \rho_i(A_j)$ for $j\in\{i,\ldots,m\}$, $F_j':=\rho_i(F_j)$ for $j\in\{i,\ldots,m\}\cup\{\infty\}$, and $\rho_i$ is the affine reflection in the hyperplane containing $F_i$. Note that any two folding operators commute. An index $j$ such that $A_{j-1}=A_j$ is called a {\em folding position} of $\Omega$. Let $\mbox{fp}(\Omega):=\{ j_1<\ldots< j_s\}$ be the set of folding positions of $\Omega$. If this set is empty, $\Omega$ is called {\em unfolded}. Given this data, we define the operator ``unfold'', producing an unfolded alcove walk, by
\[\mbox{unfold}(\Omega)=\phi_{j_1}\ldots \phi_{j_s} (\Omega)\,.\]

\begin{definition} 
A folding position $j$ of the alcove walk $\Omega=(A_0,F_1,A_1, F_2, \ldots , F_m, A_m, F_{\infty})$ is called a {\em positive folding} if the alcove $A_{j-1}=A_j$ lies on the positive side of the affine hyperplane containing the face $F_j$. Otherwise, the folding position is called a {\em negative folding}.
\end{definition}

We now fix a dominant weight $\lambda$ and a reduced alcove path $\Pi:=(A_0,A_1,\ldots,A_m)$ from $A_\circ=A_0$ to the alcove $A_m$ of minimum length in the $W$-orbit of the translate $A_\circ + \lambda$ (under the bijection between alcoves and affine Weyl group elements). Assume that we have
\[A_0\stackrel{\beta_1}\longrightarrow A_1\stackrel{\beta_2}\longrightarrow \ldots
\stackrel{\beta_m}\longrightarrow A_{m}\,,\]
where $\Gamma:=(\beta_1,\ldots,\beta_m)$ is a sequence of positive roots. This sequence, which determines the alcove path, is called a {\em $\lambda$-chain} (of roots). 

\begin{remark} $\lambda$-chains were defined in \cite{lapcmc,lapawg} based on alcove paths from $A_\circ$ to $A_\circ-\lambda$. Two equivalent definitions of such $\lambda$-chains (in terms of reduced words in affine Weyl groups, and an interlacing condition) can be found in \cite{lapawg}[Definition 5.4] and \cite{lapcmc}[Definition 4.1 and Proposition 4.4]. Hence, the $\lambda$-chains considered in this paper  are obtained by reversing the ones in the mentioned papers and by removing a certain segment at the end. The reason for this removal is that the alcove paths here do not end at the alcove $A_\circ + \lambda$, but at the minimum length representative in its orbit.
\end{remark}

 We also let $r_i:=s_{\beta_i}$, and let $\widehat{r}_i$ be the affine reflection in the common wall of $A_{i-1}$ and $A_i$, for $i=1,\ldots,m$; in other words, $\widehat{r}_i:=s_{\beta_i,l_i}$, where $l_i:=|\{j\le i\::\: \beta_j = \beta_i\}|$ is the cardinality of the corresponding set. Given $J=\{j_1<\ldots<j_s\}\subseteq[m]:=\{1,\ldots,m\}$, we define the Weyl group element $\phi(J)$ and the weight $\mu(J)$ by
\begin{equation}\label{defphimu}\phi(J):={r}_{j_1}\ldots {r}_{j_s}\,,\;\;\;\;\;\mu(J):=\widehat{r}_{j_1}\ldots \widehat{r}_{j_s}(\lambda)\,.\end{equation}

 Given $w\in W$, we define the alcove path $w(\Pi):=(w(A_0),w(A_1),\ldots,w(A_m))$. Consider the set of alcove paths
\[{\mathcal P}(\Gamma):=\{w(\Pi)\::\:w\in W\}\,.\]
We identify any $w(\Pi)$ with the obvious unfolded alcove walk of weight $\mu(w(\Pi)):=w(\lambda)$. Let us now consider the set of alcove walks
\[{\mathcal F}(\Gamma):=\{\,\mbox{alcove walks $\Omega$}\::\:\mbox{unfold}(\Omega)\in{\mathcal P}(\Gamma)\}\,.\]
We can encode an alcove walk $\Omega$ in ${\mathcal F}(\Gamma)$ by the pair $(w,J)$ in $W\times 2^{[m]}$, where 
\[\mbox{fp}(\Omega)=J\;\;\;\;\mbox{and}\;\;\;\;\mbox{unfold}(\Omega)=w(\Pi)\,.\]
Clearly, we can recover $\Omega$ from $(w,J)$ with $J=\{j_1<\ldots<j_s\}$ by
\[\Omega=\phi_{j_1}\ldots \phi_{j_s} (w(\Pi))\,.\]
We call a pair $(w,J)$ a {\em folding pair}, and, for simplicity, we denote the set $W\times 2^{[m]}$ of such pairs by ${\mathcal F}(\Gamma)$ as well. Given a folding pair $(w,J)$, the corresponding positive and negative foldings (viewed as a partition of $J$) are denoted by $J^+$ and $J^-$.   

\begin{proposition}\label{admpairs} {\rm (1)} Consider a folding pair $(w,J)$ with $J=\{j_1<\ldots<j_s\}$. We have $j_i\in J^+$ if and only if 
\[wr_{j_1}\ldots r_{j_{i-1}}>wr_{j_1}\ldots r_{j_{i-1}}r_{j_{i}}\,.\]

{\rm (2)} If $\Omega\mapsto (w,J)$, then
\[\mu(\Omega)=w(\mu(J))\,.\] 
\end{proposition}

\begin{proof} The first part rests on the well-known fact that, given a positive root $\alpha$ and a Weyl group element $w$, we have $\ell(ws_\alpha)<\ell(w)$ if and only if $w(\alpha)$ is a negative root \cite[Proposition~5.7]{humrgc}. The second part follows from the simple fact that the action of $w$ on an alcove walk commutes with that of the folding operators. 
\end{proof}

We call the sequence
\[w,\,wr_{j_1},\,\ldots,\,wr_{j_1}\ldots r_{j_{s}}=\phi(J)\]
the Bruhat chain associated to $(w,J)$. 

We now restate the Ram-Yip formula \cite{raycfm} for the Macdonald polynomials $P_\lambda(X;q,t)$ in terms of folding pairs. Since in this paper we only consider weights $\lambda$ that are regular (and dominant), we make this assumption from now on. As stated above, the $\lambda$-chain $\Gamma$ is fixed. 

\begin{theorem}\cite{raycfm} \label{hlpthm} We have
\begin{align}\label{hlpform}&P_{\lambda}(X;q,t)=\\
&\sum_{(w,J)\in{\mathcal F}(\Gamma)}t^{\frac{1}{2}(\ell(w)-\ell(w\phi(J))-|J|)}\,(1-t)^{|J|}\left(\prod_{j\in J^+}\frac{1}{1-q^{l_j}t^{\langle\rho,\beta_j^\vee\rangle}}\right)\left(\prod_{j\in J^-}\frac{q^{l_j}t^{\langle\rho,\beta_j^\vee\rangle}}{1-q^{l_j}t^{\langle\rho,\beta_j^\vee\rangle}}\right)x^{w(\mu(J))}\,.\nonumber\end{align}
\end{theorem}

\subsection{A new formula of Haglund-Haiman-Loehr type} 

In this subsection we present a new formula for the Macdonald polynomials of type $A$ that is similar to the Haglund-Haiman-Loehr one \cite{hhlcfm}. This formula will be derived by compressing the Ram-Yip formula. It also turns out that the new formula has considerably fewer terms even than the Haglund-Haiman-Loehr formula (cf. Section \ref{compry}). 

Let $\lambda = (\lambda _{1}\geq \lambda _{2}\geq \ldots \geq \lambda _{l})$ with $\lambda_l>0$ be a
partition of $m = \lambda _{1}+\ldots +\lambda _{l}$. The number of parts $l$ is known as the {\em length} of $\lambda$, and is denoted by $\ell(\lambda)$. Using standard notation, one defines
\[n(\lambda):=\sum_{i}(i-1)\lambda_i\,.\]
We identify $\lambda$ with its Young (or Ferrers) diagram
\begin{equation*}
 \{(i,j)\in  \Z _{+}\times \Z _{+}: j\leq \lambda _{i}\}\,,
\end{equation*}
whose elements are called {\it
cells}.   Diagrams are drawn in ``Japanese style'' (i.e., in the third quadrant), as shown below:
\begin{equation*}
\lambda =(2,2,2,1) = \tableau{{}&{}\\ {}&{}\\{}&{}\\&{}}\;;
\end{equation*}
the rows and columns are increasing in the negative direction of the axes. We denote, as usual, by $\lambda'$ the conjugate partition of $\lambda$ (i.e., the reflection of the diagram of $\lambda$ in the line $y=-x$, which will be drawn in French style). For any cell $u=(i,j)$ of $\lambda$ with $j\ne 1$, denote the cell $v=(i,j-1)$ directly to the right of $u$ by $\rt(u)$. 

Two cells $u,v\in \lambda $ are said to {\it attack} each other if either
\begin{itemize}
\item [(i)] they are in the same column: $u = (i,j)$, $v = (k,j)$; or
\item [(ii)] they are in consecutive columns, with the cell in the left column
strictly above the one in the right column: $u=(i,j)$,
$v=(k,j-1)$, where $i<k$.
\end{itemize}
The figure below shows the two types of pairs of attacking cells.
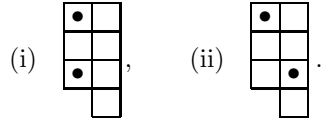
\begin{equation*}
\text{(i)}\quad \tableau{{\bullet}&{}\\ {}&{}\\{\bullet}&{}\\&{}} \, ,\qquad
\text{(ii)}\quad  \tableau{{\bullet}&{}\\ {}&{}\\{}&{\bullet}\\&{}}\; .
\end{equation*}

\begin{remark} The main difference in our approach compared to the Haglund-Haiman-Loehr one is in the definition of attacking cells; note that in \cite{hhlcfm} these cells are defined similarly, except that $u=(i,j)$ and $v=(k,j-1)$ with  $i>k$ attack each other.
\end{remark} 

A {\it filling} is a function
$\sigma \::\: \lambda \rightarrow [n]:=\{1,\ldots,n\}$ for some $n$, that is, an assignment of values in $[n]$ to the cells of $\lambda$.  As usual, we define the content of a filling $\sigma$ as ${\rm content}(\sigma):=(c_1,\ldots,c_n)$, where $c_i$ is the number of entries $i$ in the filling, i.e., $c_i:=|\sigma^{-1}(i)|$. The monomial $x^{{\rm content}(\sigma)}$ of degree $m$ in the variables $x_{1},\ldots,x_n$ is then given by
\begin{equation*}\label{e:xsigma}
x^{{\rm content}(\sigma)} := x_1^{c_1}\ldots,x_n^{c_n}\,.
\end{equation*}

\begin{definition}\label{deff} A filling $\sigma\::\:\lambda\rightarrow [n]$ is called {\em nonattacking} if $\sigma(u)\ne\sigma(v)$ whenever  $u$ and $v$ attack each other. 
Let ${\mathcal T}(\lambda,n)$ denote the set of nonattacking fillings.
\end{definition}

\begin{definition} Given a filling $\sigma$ of $\lambda$, let
\[\Des(\sigma):=\{(i,j)\in\lambda\::\:(i,j+1)\in\lambda\,,\;\;\sigma(i,j)>\sigma(i,j+1)\}\,.\]
Also let 
\[\Diff(\sigma):=\{(i,j)\in\lambda\::\:(i,j+1)\in\lambda\,,\;\;\sigma(i,j)\ne\sigma(i,j+1)\}\,.\]
\end{definition}

We define a reading order   on the cells of $\lambda$ as  the total order
given by reading each column from top to bottom, and by considering the columns from right to left (smallest to largest). Note that this is a different reading order than the usual (French or Japanese) ones.

\begin{definition} An {\it
inversion} of $\sigma $ is a pair of entries $\sigma (u)>\sigma (v)$,
where $u$ and $v$ attack each other, and $u$ precedes $v$ in the
considered reading order. Let ${\rm Inv}(\sigma)$ denote the set of inversions of $\sigma$.
\end{definition}

Here are two examples of inversions, where $a<b$:
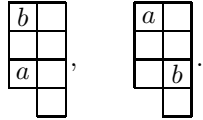
\begin{equation*}
\tableau{{b}&{}\\ {}&{}\\{a}&{}\\&{}} \, ,\qquad 
\tableau{{a}&{}\\ {}&{}\\{}&{b}\\&{}}\,.
\end{equation*}

The {\em arm} of a cell $u\in\lambda$ is the number of cells strictly to the left of $u$ in the same row; similarly, the {\em leg} of $u$ is the number of cells strictly below $u$ in the same column, as illustrated below.
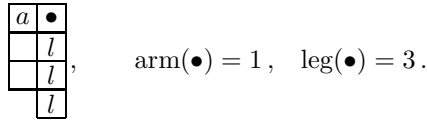
\begin{equation*}
\tableau{{a}&{\bullet}\\ {}&{l}\\{}&{l}\\&{l}} \, ,\qquad \arm(\bullet)=1\,,\;\;\;\leg(\bullet)=3\,.
\end{equation*}

\begin{definition} The {\em maj statistic} on fillings $\sigma$ is defined by
\[\maj(\sigma)=\sum_{u\in\Des(\sigma)}\arm(u)\,.\]
The {\em inversion statistic}  is defined by 
\[\inv(\sigma)=|{\rm Inv}(\sigma)|-\sum_{u\in\Des(\sigma)}\leg(u)\,.\]
\end{definition}

We are now ready to state a new combinatorial formula for the Macdonald $P$-polynomials in the variables $X=(x_1,\ldots,x_n)$ for a fixed $n$ and a partition $\lambda$ which corresponds to a regular weight, that is, $(\lambda_1>\ldots>\lambda_{n-1}>0)$. 

\begin{theorem}\label{hlqthm}
Given $\lambda$ as above, we have
\begin{equation}\label{hlqform}
P_{\lambda}(X;q,t) = \sum_{\sigma\in{\mathcal T}(\lambda,n)}
 t^{n(\lambda) - \inv (\sigma)}q^{\maj(\sigma)}\left(\prod_{u\in\Diff(\sigma)}\frac{1-t}{1-q^{\arm(u)}t^{\leg(u)+1}}\right) x^{{\rm content}(\sigma) } \,.
\end{equation}
\end{theorem}

\section{The compression phenomenon}

\subsection{Specializing the Ram-Yip formula to type $A$}\label{specschwer} We now restrict ourselves to the root system of type $A_{n-1}$, fow which the Weyl group $W$ is the symmetric group $S_n$. Permutations $w\in S_n$ are written in one-line notation $w=w(1)\ldots w(n)$. 
We can identify the space $\h_\R^*$ with the quotient space 
$V:=\R^n/\R(1,\ldots,1)$,
where $\R(1,\ldots,1)$ denotes the subspace in $\R^n$ spanned 
by the vector $(1,\ldots,1)$.  
The action of the symmetric group $S_n$ on $V$ is obtained 
from the (left) $S_n$-action on $\R^n$ by permutation of coordinates.
Let $\varepsilon_1,\ldots,\varepsilon_n\in V$ 
be the images of the coordinate vectors in $\R^n$.
The root system $\Phi$ can be represented as 
$\Phi=\{\alpha_{ij}:=\varepsilon_i-\varepsilon_j \::\: i\ne j,\ 1\leq i,j\leq n\}$.
The simple roots are $\alpha_i=\alpha_{i,i+1}$, 
for $i=1,\ldots,n-1$.
The fundamental weights are $\omega_i = \varepsilon_1+\ldots +\varepsilon_i$, 
for $i=1,\ldots,n-1$. 
The weight lattice is $\Lambda=\Z^n/\Z(1,\ldots,1)$. A dominant weight $\lambda=\lambda_1\varepsilon_1+\ldots+\lambda_{n-1}\varepsilon_{n-1}$ is identified with the partition $(\lambda _{1}\geq \lambda _{2}\geq \ldots \geq \lambda _{n-1}\geq\lambda_n=0)$ of length at most $n-1$. We fix such a partition $\lambda$ for the remainder of this paper, and assume that the corresponding weight is regular, i.e., $(\lambda _{1}> \lambda _{2}> \ldots > \lambda _{n-1} >\lambda_n=0)$. 

For simplicity, we use the same notation $(i,j)$ with $i<j$ for the root $\alpha_{ij}$ and the reflection $s_{\alpha_{ij}}$, which is the transposition of $i$ and $j$.  Consider the following chain of roots, denoted by $\Gamma(k)$:
\begin{equation}\label{omegakchain}\begin{array}{lllll}
(&\!\!\!\!(k,n),&(k,n-1),&\ldots,&(k,k+1)\,,\\
&\!\!\!\!(k-1,n),&(k-1,n-1),&\ldots,&(k-1,k+1)\,,\\
&&&\ldots\\
&\!\!\!\!(1,n),&(1,n-1),&\ldots,&(1,k+1)\,\,)\,.
\end{array}\end{equation}
Denote by $\Gamma'(k)$ the chain of roots obtained by removing the root $(i,k+1)$ at the end of each row. Now define a chain $\Gamma$ as a concatenation $\Gamma:=\Gamma_{\lambda_1}\ldots\Gamma_2$, where 
\[\Gamma_j:=\casetwo{\Gamma'(\lambda'_j)}{\mbox{$j=\min\;\{i\::\:\lambda_i'=\lambda_j'\}$}}{\Gamma(\lambda'_j)}\]
Based on the interlacing condition in \cite{lapcmc}[Definition 4.1 and Proposition 4.4], it is not hard to verify that $\Gamma$ is a $\lambda$-chain in the sense defined in Section \ref{alcovewalks}. 

Alternatively, one can argue based on the chain considered in \cite{lenhlp}[Section 3.1], which we denote here by $\widehat{\Gamma}$. We proved in \cite[Corollary 15.4]{lapawg} that, for any $k=1,\ldots,n-1$,
we have the following chain of roots corresponding to an alcove path from $A_\circ$ to $A_{\circ}+\omega_k$, denoted by $\widehat{\Gamma}(k)$:
\begin{equation}\begin{array}{lllll}
(&\!\!\!\!(1,n),&(1,n-1),&\ldots,&(1,k+1)\,,\\
&\!\!\!\!(2,n),&(2,n-1),&\ldots,&(2,k+1)\,,\\
&&&\ldots\\
&\!\!\!\!(k,n),&(k,n-1),&\ldots,&(k,k+1)\,\,)\,.
\end{array}\end{equation}
Hence, we can construct a chain corresponding to an alcove path from $A_\circ$ to $A_\circ+\lambda$ as a concatenation $\widehat{\Gamma}:=\widehat{\Gamma}_{\lambda_1}\ldots\widehat{\Gamma}_1$, where $\widehat{\Gamma}_j=\widehat{\Gamma}(\lambda'_j)$. It is not hard to show that we can use moves of the form $((i,j),(i,k),(j,k))\leftrightarrow((j,k),(i,k),(i,j))$ with $i<j<k$ to change $\widehat{\Gamma}$ into a chain/alcove path which has $\Gamma$ as an initial segment (the mentioned moves translate into Coxeter moves for the reduced words in the affine Weyl group corresponding to the alcove paths); moreover, the reflections in the final segment, when applied from left to right, give rise to a saturated chain in Bruhat order on $S_n$ from the identity to the longest permutation. 

The $\lambda$-chain $\Gamma$ is fixed for the remainder of this paper. Thus, we can replace the notation ${\mathcal F}(\Gamma)$ with ${\mathcal F}(\lambda)$.

\begin{example}\label{ex21} {\rm Consider $n=4$ and $\lambda =(4,3,1,0)$, for which we have the following $\lambda$-chain (the underlined pairs are only relevant in Example \ref{ex21c} below):
\begin{equation}\label{exlchain}\Gamma=\Gamma_4\Gamma_3\Gamma_2=(\underline{(1,4)},(1,3)\:|\:(2,4),\underline{(2,3)},(1,4),\underline{(1,3)}\:|\:\underline{(2,4)},(1,4))\,.\end{equation}
We represent the Young diagram of $\lambda$ inside a broken $4\times 4$ rectangle, as shown below. In this way, a transpositions $(i,j)$ in $\Gamma$ can be viewed as swapping entries in the two parts of each column (in rows $i$ and $j$, where the row numbers are also indicated below). 
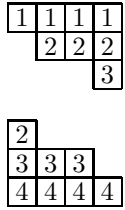
\begin{equation*}
 \begin{array}{l} \tableau{{1}&{1}&{1}&{1}\\ &{2}&{2}&{2}\\&&&{3}}\\ \\
\tableau{{2}\\ {3}&{3}&{3}\\ {4}&{4}&{4}&{4}} \end{array}
\end{equation*}
}
\end{example}

Given the $\lambda$-chain $\Gamma$ above, in Section \ref{alcovewalks} we considered subsets $J=\{ j_1<\ldots< j_s\}$ of $[m]$, where $m$ is the length of the $\lambda$-chain. Instead of $J$, it is now convenient to use the subsequence of $\Gamma$ indexed by the positions in $J$. This is viewed as a concatenation with distinguished factors $T=T_{\lambda_1}\ldots T_2$ induced by the factorization of $\Gamma$ as $\Gamma_{\lambda_1}\ldots\Gamma_2$. The partition $J=J^+\sqcup J^-$ induces partitions $T=T^+\sqcup T^-$ and $T_j=T_j^+\sqcup T_j^-$.

All the notions defined in terms of $J$ are now redefined in terms of $T$. As such, from now on we will write  $\phi(T)$, $\mu(T)$, and $|T|$, the latter being the size of $T$. If $(w,J)$ is a folding pair, we will also call the corresponding $(w,T)$ a folding pair. We will use the notation ${\mathcal F}(\Gamma)$ and ${\mathcal F}(\lambda)$ accordingly.  

We denote by $wT_{\lambda_1}\ldots T_{j}$ the permutation obtained from $w$ via right multiplication by the transpositions in $T_{\lambda_1},\ldots, T_{j}$, considered from left to right. This agrees with the above convention of using pairs to denote both roots and the corresponding reflections. As such, $\phi(J)$ in (\ref{defphimu}) can now be written simply $T$. 

\begin{example}\label{ex21c}{\rm We continue Example \ref{ex21}, by picking the folding pair $(w,J)$ with $w=2341\in S_4$ and $J=\{1,4,6,7\}$ (see the underlined positions in (\ref{exlchain})). Thus, we have
\[T=T_4T_3T_2=((1,4)\:|\:(2,3),(1,3)\:|\:(2,4))\,.\]
Note that $J^+=\{1,7\}$ and $J^-=\{4,6\}$. Indeed, we have the corresponding chain in Bruhat order, where the swapped entries are shown in bold (we represent permutations as broken columns, as discussed in Example \ref{ex21}):
\[w=\begin{array}{l}\tableau{{{\mathbf 2}}} \\ \\ \tableau{{3}\\{4}\\{{\mathbf 1}}} \end{array} \:>\:\begin{array}{l}\tableau{{1}} \\ \\ \tableau{{3}\\{4}\\{2}} \end{array}\:|\: \begin{array}{l}\tableau{{{ 1}}\\{{\mathbf 3}}} \\ \\ \tableau{{{\mathbf 4}}\\{2}} \end{array}\:<\: \begin{array}{l}\tableau{{{\mathbf 1}}\\{{4}}} \\ \\ \tableau{{{\mathbf 3}}\\{2}}\end{array}\:<\: \begin{array}{l}\tableau{{3}\\{4}}\\ \\ \tableau{{1}\\{2}}\end{array}\:|\: 
\begin{array}{l}\tableau{{3}\\{{\mathbf 4}}\\{1}}\\ \\ \tableau{{{\mathbf 2}}}\end{array} 
\:>\: \begin{array}{l} \tableau{{3}\\{2}\\{1}} \\ \\ \tableau{{4}}
\end{array} \,.\]
}
\end{example}

\subsection{The filling map} Given a folding pair $(w,T)$, we consider the permutations
\[\pi_j=\pi_j(w,T):=wT_{\lambda_1}T_{\lambda_1-1}\ldots T_{j+1}\,,\]
for $j=1,\ldots,\lambda_1$. In particular, $\pi_{\lambda_1}=w$. 

\begin{definition}\label{deffill}
The {\em filling map} is the map $f$ from folding pairs $(w,T)$ to fillings $\sigma=f(w,T)$ of the shape $\lambda$, defined by
\begin{equation}\label{defswt}\sigma(i,j):=\pi_j(i)\,.\end{equation}
In other words, the $j$-th column of the filling $\sigma$ consists of the first $\lambda_j'$ entries of the permutation $\pi_j$.
\end{definition}

\begin{example} {\rm Given $(w,T)$ as in Example \ref{ex21c}, we have
\[f(w,T)=\tableau{{2}&{1}&{3}&{3}\\&{3}&{4}&{2}\\&&&{1}}\,.\]}
\end{example}

The following proposition is a similar version of \cite{lenhlp}[Proposition 3.6].

\begin{proposition}\label{weightmon} For permutation $w$ and any subsequence $T$ of $\Gamma$, we have ${\rm content}(f(w,T))=w(\mu(T))$. In particular, $w(\mu(T))$ only depends on $f(w,T)$. \end{proposition}

\begin{proposition}\label{surjmap}
We have $f(\mathcal{F}(\lambda))\subseteq{\mathcal T}(\lambda,n)$. If the partition $\lambda$ corresponds to a regular weight, then the map $f\::\:{\mathcal F}(\lambda)\rightarrow{\mathcal T}(\lambda,n)$ is surjective. 
\end{proposition}

\begin{proof}
 Let $u=(i,j)$ be a cell of $\lambda$, and $\sigma=f(w,T)$. We check that $\sigma$ satisfies the condition in Definition \ref{deff}. Let $v=(k,j)$ with $k<i$. We clearly have $\sigma(u)\ne\sigma(v)$ because $\sigma(u)=\pi_j(i)$ and $\sigma(v)=\pi_j(k)$. For the same reason, if $\sigma(u)=\sigma(\rt(u))$, then $\sigma(v)\ne\sigma(\rt(u))$. Otherwise, consider the subchain of the Bruhat chain corresponding to $(w,T)$ which starts at $\pi_j$ and ends at $\pi_{j-1}$. There is a permutation $\pi$ in this subchain such that $\sigma(v)=\pi(k)$ and $\sigma(\rt(u))$ is the entry in some position greater than $k$, to be swapped with position $i$; this follows from the structure of the segment $\Gamma_j$ (see (\ref{omegakchain})) in the $\lambda$-chain $\Gamma$. Thus, we have $\sigma(v)\ne\sigma(\rt(u))$ once again. We conclude that $\sigma\in{\mathcal T}(\lambda,n)$. 

Now consider $\lambda$ corresponding to a regular weight, and $\sigma\in{\mathcal T}(\lambda,n)$. We construct a chain in the Bruhat order on $S_n$, as follows. Let $\pi_1$ be the unique permutation such that $\pi_1(i)=\sigma(i,1)$ for $1\le i\le\lambda_1'$. Assume that we constructed the Bruhat chain up to $\pi_{j-1}$. For each $i$ from 1 to $\lambda_j'$, if $\sigma(i,j)\ne\sigma(i,j-1)$, then swap the entry in position $i$ of the current permutation with the entry $\sigma(i,j)$; the latter is always found in a position greater than $\lambda_{j-1}'$ because $\sigma\in{\mathcal T}(\lambda,n)$. The result is a permutation $\pi_j$ whose first $\lambda_j'$ entries form column $j$ of $\sigma$. Continue in this way up to column $\lambda_1$. Then set $w:=\pi_{\lambda_1}$. The obtained Bruhat chain determines a folding pair $(w,T)$ mapped to $\sigma$ by $f$. 
\end{proof}

Based on Proposition \ref{surjmap}, from now on we consider the filling map as a map $f\::\:{\mathcal F}(\lambda)\rightarrow{\mathcal T}(\lambda,n)$. 

\subsection{Compressing the Ram-Yip formula}\label{compry} From now on, we assume that the partition $\lambda$  corresponds to a regular weight. 

We start by rewriting the Ram-Yip formula (\ref{hlpform}) in the type $A$ setup and by recalling our new formula (\ref{hlqform}) in terms of fillings:
\begin{align*}
& P_{\lambda}(X;q,t)=\sum_{(w,T)\in{\mathcal F}(\Gamma)}t^{\frac{1}{2}(\ell(w)-\ell(wT)-|T|)}\,(1-t)^{|T|}\!\!\left(\prod_{j,(i,k)\in T_j^+}\frac{1}{1-q^{\arm(i,j-1)}t^{k-i}}\right)\times\\
&\;\;\;\;\;\;\;\;\;\;\;\;\;\;\;\;\;\;\;\;\;\;\times\left(\prod_{j,(i,k)\in T_j^-}\frac{q^{\arm(i,j-1)}t^{k-i}}{1-q^{\arm(i,j-1)}t^{k-i}}\right)x^{w(\mu(T))}\,,\\
& P_{\lambda}(X;q,t) = \sum_{\sigma\in{\mathcal T}(\lambda,n)}
 t^{n(\lambda) - \inv (\sigma)}q^{\maj(\sigma)}\left(\prod_{u\in\Diff(\sigma)}\frac{1-t}{1-q^{\arm(u)}t^{\leg(u)+1}}\right) x^{{\rm content}(\sigma) } \,.
\end{align*}

We will now describe the way in which the second formula can be obtained by compressing the first one.

\begin{theorem}\label{mainthm}
Given any $\sigma \in{\mathcal T}(\lambda,n)$, we have $f^{-1}(\sigma)\ne\emptyset$ and $x^{w(\mu(T))}=x^{{\rm content}(\sigma)}$ for all $(w,T)\in f^{-1}(\sigma)$. Furthermore, we have
\begin{align*}&\sum_{(w,T)\in f^{-1}(\sigma)}\!\!\!\!\!\!\!\!\!t^{\frac{1}{2}(\ell(w)-\ell(wT)-|T|)}\,(1-t)^{|T|}\left(\prod_{j,(i,k)\in T_j^+}\frac{1}{1-q^{\arm(i,j-1)}t^{k-i}}\right)\left(\prod_{j,(i,k)\in T_j^-}\frac{q^{\arm(i,j-1)}t^{k-i}}{1-q^{\arm(i,j-1)}t^{k-i}}\right)= \\&=t^{n(\lambda) - \inv (\sigma)}q^{\maj(\sigma)}\left(\prod_{u\in\Diff(\sigma)}\frac{1-t}{1-q^{\arm(u)}t^{\leg(u)+1}}\right) \,.\end{align*}
\end{theorem}

The first statement in Theorem \ref{mainthm} is just the content of Propositions \ref{weightmon} and \ref{surjmap}. The compression formula in the theorem  will be proved Section \ref{pfmain}. Then Theorem \ref{hlqthm} becomes a corollary.

In order to measure the compression phenomenon, we define the {\em compression factor} $c(\lambda)$ as the ratio of the number of terms in the Ram-Yip formula for $\lambda$ and the number of terms $t(\lambda)$ in our new formula (\ref{hlqform}); note that we have $2^mn!$ terms in the Ram-Yip formula, where $m$ is the length of the $\lambda$-chain. We also compute the ratio $r(\lambda)$ of the number of terms in the Haglund-Haiman-Loehr formula and $t(\lambda)$. We list below some examples.

\begin{equation*}
\begin{tabular}{|c|c|r|r|r|}
\hline
$\lambda$ &  $n$ & $t(\lambda)$ & $c(\lambda)$ & $r(\lambda)$
 \\ \hline
(3, 2, 1, 0) & $4$ & 288 & 1.3 & 3.0
 \\ \hline
(5, 3, 1, 0) & $4$ & 10,368  & 4.7 & 3.0
 \\ \hline
(4, 3, 2, 1, 0) & $5$ & 34,560  &  3.6 & 7.5
 \\ \hline
(5, 4, 2, 1, 0) & $5$ & 552,960  &  14.2 & 7.5
 \\ \hline
\end{tabular}
\end{equation*}

We note that the compression factor increases with the rank of the root system and the number of columns of $\lambda$. Also note that the Haglund-Haiman-Loehr formula has more terms than the new formula, and sometimes even more than the Ram-Yip formula.

\section{The proof of Theorem \ref{mainthm}}\label{pfmain}

We start with some notation related to sequences of positive integers. Given such a sequence $w$, we write $w[i,j]$ for the subsequence $w(i)w(i+1)\ldots w(j)$. We use the notation $N_a(w)$ and $N_{ab}(w)$ for the number of entries $w(i)$ with $w(i)<a$ and $a<w(i)<b$, respectively.  

Let us denote by $\rev(S)$ the reverse of the sequence $S$. For simplicity, we write $\Gamma^r(k)$ for $\rev(\Gamma(k))$. We also consider the segment of $\Gamma^r(k)$ with the first $p$ entries removed, which we denote by $\Gamma^r(k,p)$. 

\begin{proposition}\label{p2cols0}
Consider a permutation $w$ in $S_n$ and a number $b\in[n]\setminus\{a\}$, where $a:=w(1)$; also  consider an integer $p$ with $0\le p<w^{-1}(b)-1$. Then we have
\begin{align}\label{sum2cols0}
&\stacksum{T\::\:(w,T)\in{\mathcal F}(\Gamma^r(1,p))}{wT(1)=b} t^{\frac{1}{2}(\ell(wT)-\ell(w)-|T|} (1-t)^{|T|}\left(\prod_{(1,k)\in T^+}\frac{1}{1-qt^{k-1}}\right)\left(\prod_{(1,k)\in T^-}\frac{qt^{k-1}}{1-qt^{k-1}}\right) \\
&=\case{${t^{N_{ab}(w[2,p+1])}(1-t)}/{(1-qt^{p+1})}$}{$a<b$}{${qt^{p-N_{ba}(w[2,p+1])}(1-t)}/{(1-qt^{p+1})}$}\nonumber
\end{align}
\end{proposition}

\begin{proof}
We use decreasing induction on $p$. The base case for $p=w^{-1}(b)-2$ is based on the fact that
\begin{equation}\label{basec}\frac{1}{2}(\ell(w(1,p+2))-\ell(w)-1)=\casetwo{N_{ab}(w[2,p+1])}{a<b}{-1-N_{ab}(w[2,p+1])}\end{equation}
Let us now prove the statement for $p$ assuming it for $p+1$. Let $c:=w(p+2)$. The sum in (\ref{sum2cols0}), denoted by $S(w,p)$, splits into two sums, depending on $(1,p+2)\not\in T$ and $(1,p+2)\in T$. Let us assume first that $a<b$. For simplicity, we write $N_{rs}$ for $N_{rs}(w[2,p+1])$. 

{\em Case {\rm 1}.} $a<c<b$. By induction, the first sum is
\[S(w,p+1)=\frac{t^{N_{ab}+1}(1-t)}{1-qt^{p+2}}\,.\]
By (\ref{basec}), the second sum is
\begin{align*}\frac{t^{N_{ac}}(1-t)}{1-qt^{p+1}}\,S(w(1,p+2),p+1)&=\frac{t^{N_{ac}}(1-t)}{1-qt^{p+1}}\,\frac{t^{N_{cb}}(1-t)}{1-qt^{p+2}}\,.\end{align*}
where the first equality is obtained by induction. The desired result easily follows by adding the two sums into which $S(w,p)$ splits, using $N_{ac}+N_{cb}=N_{ab}$. 

{\em Case {\rm 2}.} $a<b<c$. By induction, the first sum is
\[S(w,p+1)=\frac{t^{N_{ab}}(1-t)}{1-qt^{p+2}}\,.\]
By (\ref{basec}), the second sum is
\begin{align*}\frac{t^{N_{ac}}(1-t)}{1-qt^{p+1}}\,S(w(1,p+2),p+1)&=\frac{t^{N_{ac}}(1-t)}{1-qt^{p+1}}\,\frac{qt^{p+1-N_{bc}}(1-t)}{1-qt^{p+2}}\,.\end{align*}
where the first equality is obtained by induction. The desired result easily follows by adding the two sums into which $S(w,p)$ splits, using $N_{ac}-N_{bc}=N_{ab}$. 

{\em Case {\rm 3}.} $c<a<b$. By induction, the first sum is the same as in Case 2. By (\ref{basec}), the second sum is
\begin{align*}\frac{t^{-1-N_{ca}}(1-t)qt^{p+1}}{1-qt^{p+1}}\,S(w(1,p+2),p+1)&=\frac{t^{-1-N_{ca}}(1-t)qt^{p+1}}{1-qt^{p+1}}\,\frac{t^{N_{cb}+1}(1-t)}{1-qt^{p+2}}\,.\end{align*}
where the first equality is obtained by induction. The desired result easily follows by adding the two sums into which $S(w,p)$ splits, using $N_{cb}-N_{ca}=N_{ab}$. 

We also have three cases corresponding to $a>b$, which are verified in a completely similar way.
\end{proof}

\begin{proposition}\label{twocol} Consider two sequences $C_1,C_2$ of size $p$ and entries in $[n]$, as well as a permutation $w$ in $S_n$ such that $w[1,p]=C_1$. Let $C_2C_1$ denote the two-column filling with left column $C_2$ and right column $C_1$. Assume that $C_2C_1$ is nonattacking.  Then we have
\begin{align}\label{sum2cols}
&\stacksum{T\::\:(w,T)\in{\mathcal F}(\Gamma^r(p))}{wT[1,p]=C_2} t^{\frac{1}{2}(\ell(wT)-\ell(w)-|T|} (1-t)^{|T|}\left(\prod_{(i,k)\in T^+}\frac{1}{1-q_it^{k-i}}\right)\left(\prod_{(i,k)\in T^-}\frac{q_it^{k-i}}{1-q_it^{k-i}}\right) \\
&=t^{\binom{|C_1|}{2}-\inv(C_2C_1)+\inv(C_2)}\left(\prod_{(i,1)\in\Des(C_2C_1)}q_i\right)\left(\prod_{(i,1)\in\Diff(C_2C_1)}\frac{1-t}{1-q_it^{\leg(i,1)+1}}\right)\,.\nonumber
\end{align}
The same result holds if $C_1$ has size $p+1$ instead, and $\Gamma^r(p)$ is replaced by $\rev(\Gamma'(p))$.
\end{proposition}

\begin{proof}
We introduce some notation first. Consider the following two conditions (related to a descent) for a pair $(i,k)$ with $1\le i<k\le p$:
\begin{itemize}
\item $C_1(i)>C_1(k)$;
\item $C_1(k)>C_2(i)$.
\end{itemize}
Let $N_{d\wedge 1}$, $N_{1\wedge d}$, $N_{d\wedge d}$, $N_{d\vee d}$ and $N_{a\wedge a}$ denote the number of pairs $(i,k)$ satisfying the first condition, the second one, both conditions, at least one condition, and none of the conditions (meaning that both descents are replaced by ascents), respectively. Let us also define 
\[L:=\sum_{u\in\Des(C_2C_1)}\leg(u)\,.\]

Let us split $\Gamma^r(p)$ as $\Gamma_1^r(p)\ldots\Gamma_p^r(p)$ by the rows in (\ref{omegakchain}), that is,
\[\Gamma_i^r(p)=(\,(i,p+1),\:(i,p+2),\:\ldots ,\:(i,n)\,)\,.\]
This splitting induces one for the subsequence $T$ of $\Gamma^r(p)$ indexing the sum in (\ref{sum2cols}), namely $T=T_1\ldots T_p$. For $i=0,1,\ldots,p$, define $w_i:=wT_1 \ldots T_{i}$, so $w_0=w$. The sum in (\ref{sum2cols}) can be written as a $p$-fold sum in the following way:
\begin{equation}\label{kfold}\stacksum{T_p\::\:(w_{p-1},T_p)\in\F(\Gamma_p^r(p))}{w_p(p)=C_2(p)} E(w_{p-1},T_p)\;\ldots  \stacksum{T_1\::\:(w_0,T_1)\in\F(\Gamma_1^r(p))}{w_{1}(1)=C_2(1)} E(w_0,T_1)\,,\end{equation}
where 
\begin{equation*}E(w,T):=t^{\frac{1}{2}(\ell(wT)-\ell(w)-|T|} (1-t)^{|T|}\left(\prod_{(i,k)\in T^+}\frac{1}{1-q_it^{k-i}}\right)\left(\prod_{(i,k)\in T^-}\frac{q_it^{k-i}}{1-q_it^{k-i}}\right)\,.\end{equation*}
By Proposition \ref{p2cols0}, we have
\[\stacksum{T_i\::\:(w_{i-1},T_i)\in\F(\Gamma_i^r(p))}{w_{i}(i)=C_2(i)}\!\!\!\!\!\!\!\!\!\!\!\!\!\!\!\!\!E(w_{i-1},T_i)= \casethree{{t^{N_{C_1(i),C_2(i)}(C_1[i+1,p])}(1-t)}/{(1-q_it^{p-i+1})}}{C_1(i)<C_2(i)}{{q_it^{p-i-N_{C_2(i),C_1(i)}(C_1[i+1,p])}(1-t)}/{(1-q_it^{p-i+1})}}{C_1(i)>C_2(i)}{1}{C_1(i)=C_2(i)}\]
We can see that the above sum does not depend on the permutation $w_{i}$, but only on $C_1(i)$ and $C_2[i,p]$. Therefore, using the notation above, the $p$-fold sum (\ref{kfold}) evaluates to 
\[t^{N_{a\wedge a}+L-N_{d\wedge d}}\left(\prod_{(i,1)\in\Des(C_2C_1)}q_i\right)\left(\prod_{(i,1)\in\Diff(C_2C_1)}\frac{1-t}{1-q_it^{\leg(i,1)+1}}\right)\,.\]
Now observe that 
\[\binom{|C_1|}{2}-\inv(C_2C_1)-L+\inv(C_2)=\binom{|C_1|}{2}-N_{d\wedge 1}-N_{1\wedge  d}=\binom{|C_1|}{2}-N_{d\vee d}-N_{d\wedge d}=N_{a\wedge a}-N_{d\wedge d}\,.\]
This concludes the proof.

The case when $C_1$ has size $p+1$ is reduced to the previous one by extending $C_2$ to size $p+1$ via setting $C_2(p+1):=C_1(p+1)$.
\end{proof}

\begin{proof}[Proof of Theorem {\rm \ref{mainthm}}]
The splitting $\rev(\Gamma)=\Gamma_2^r\ldots \Gamma_{\lambda_1}^r$, where $\Gamma$ is our fixed $\lambda$-chain and $\Gamma_j^r:=\rev(\Gamma_j)$, induces a splitting $T=T_2\ldots T_{\lambda_1}$ of any $T$ for which $(w,T)\in\F(\rev(\Gamma))$, cf. Section \ref{specschwer}.  Let $m:=\lambda_1$ be the number of columns of $\lambda$, and let $C=C_1,\ldots,C_m$ be the columns of a fixed filling $\sigma$, of lengths $c_1:=\lambda_1',\ldots,c_m:=\lambda_m'$. For $j=1,\ldots,m$, define $w_j:=wT_2\ldots T_{j}$, so $w_1=w$. The sum in Theorem {\rm \ref{mainthm}} can be written as an $(m-1)$-fold sum in the following way:
\begin{equation}\label{mfold}\stacksum{T_{m}\::\:(w_{m-1},T_{m})\in\F(\Gamma_{m}^r)}{w_{m}[1,c_m]=C_m} E(m,w_{m-1},T_m) \;\ldots  \stacksum{T_2\::\:(w_1,T_2)\in\F(\Gamma_2^r)}{w_{2}[1,c_2]=C_2} E(2,w_1,T_2)\,;\end{equation}
here 
\begin{equation*}E(j,w,T):=t^{\frac{1}{2}(\ell(wT)-\ell(w)-|T|} (1-t)^{|T|}\left(\prod_{(i,k)\in T^+}\frac{1}{1-q^{\arm(i,j-1)}t^{k-i}}\right)\left(\prod_{(i,k)\in T^-}\frac{q^{\arm(i,j-1)}t^{k-i}}{1-q^{\arm(i,j-1)}t^{k-i}}\right)\,,\end{equation*} 
and $\arm(i,j-1)$ is computed in $\lambda$. 
By Proposition \ref{twocol}, we have
\begin{align*}&\stacksum{T_{j}\::\:(w_{j-1},T_{j})\in\F(\Gamma_{j}^r)}{w_{j}[1,c_j]=C_j} E(j,w_{j-1},T_j)=t^{\binom{c_{j-1}}{2}-\inv(C_jC_{j-1})+\inv(C_j)}\left(\prod_{(i,1)\in\Des(C_jC_{j-1})}q^{\arm(i,j-1)}\right)\times\\&\;\;\;\;\;\;\;\;\;\;\;\;\;\;\;\;\;\;\;\;\;\;\;\;\;\;\;\;\;\;\;\;\;\;\;\;\;\;\;\;\;\;\;\;\;\;\;\;\;\;\;\;\;\;\times\left(\prod_{(i,1)\in\Diff(C_jC_{j-1})}\frac{1-t}{1-q^{\arm(i,j-1)}t^{\leg(i,j-1)+1}}\right);\end{align*}
here both $\arm(i,j-1)$ and $\leg(i,j-1)$ are computed in $\lambda$. 
We can see that the above sum does not depend on the permutation $w_{j-1}$, but only on $C_{j-1}$ and $C_j$. Therefore, the $(m-1)$-fold sum (\ref{mfold}) is a product of $m-1$ factors, and evaluates to 
\[t^{n(\lambda) - \inv (\sigma)}q^{\maj(\sigma)}\left(\prod_{u\in\Diff(\sigma)}\frac{1-t}{1-q^{\arm(u)}t^{\leg(u)+1}}\right) \,.\]
Indeed, since $c_m=1$, we have
\[\sum_{j=2}^m \binom{c_{j-1}}{2}-\sum_{j=2}^m \left(\inv(C_jC_{j-1})-\inv(C_j)\right)=n(\lambda)-\inv(\sigma)\,.\]
\end{proof}


\begin{thebibliography}{10}

\bibitem{asasem}
S.~Assaf.
\newblock The {S}chur expansion of {M}acdonald polynomials.
\newblock Preprint, 2007.

\bibitem{gallsg}
S.~Gaussent and P.~Littelmann.
\newblock {LS}-galleries, the path model and {MV}-cycles.
\newblock {\em Duke Math. J.}, 127:35--88, 2005.

\bibitem{gahaha}
I.~Grojnowski and M.~Haiman.
\newblock Affine {H}ecke algebras and positivity of {LLT} and {M}acdonald
  polynomials.
\newblock Preprint, 2007.

\bibitem{hhlcfm}
J.~Haglund, M.~Haiman, and N.~Loehr.
\newblock A combinatorial formula for {M}acdonald polynomials.
\newblock {\em J. Amer. Math. Soc.}, 18:735--761, 2005.

\bibitem{humrgc}
J.~E. Humphreys.
\newblock {\em Reflection {G}roups and {C}oxeter {G}roups}, volume~29 of {\em
  Cambridge Studies in Advanced Mathematics}.
\newblock Cambridge University Press, Cambridge, 1990.

\bibitem{lenhlp}
C.~Lenart.
\newblock Hall-{L}ittlewood polynomials, alcove walks, and fillings of {Y}oung
  diagrams, {I}.
\newblock Preprint, 2008.

\bibitem{lapcmc}
C.~Lenart and A.~Postnikov.
\newblock A combinatorial model for crystals of {K}ac-{M}oody algebras.
\newblock {\tt arXiv:math.RT/0502147}.
\newblock To appear in {\em Trans. Amer. Math. Soc.}

\bibitem{lapawg}
C.~Lenart and A.~Postnikov.
\newblock Affine {W}eyl groups in {$K$}-theory and representation theory.
\newblock {\em Int. Math. Res. Not.}, pages Art. ID rnm038, 65, 2007.

\bibitem{litlrr}
P.~Littelmann.
\newblock {A Littlewood-Richardson rule for symmetrizable Kac-Moody algebras}.
\newblock {\em Invent. Math.}, 116:329--346, 1994.

\bibitem{litpro}
P.~Littelmann.
\newblock Paths and root operators in representation theory.
\newblock {\em Ann. of Math. {\rm (2)}}, 142:499--525, 1995.

\bibitem{macsft}
I.~Macdonald.
\newblock Schur functions: theme and variations.
\newblock In {\em S\'eminaire Lotharingien de Combinatoire (Saint-Nabor,
  1992)}, volume 498 of {\em Publ. Inst. Rech. Math. Av.}, pages 5--39. Univ.
  Louis Pasteur, Strasbourg, 1992.

\bibitem{macopa}
I.~Macdonald.
\newblock Orthogonal polynomials associated with root systems.
\newblock {\em S\'em. Lothar. Combin.}, 45:Art.\ B45a, 40 pp. (electronic),
  2000/01.

\bibitem{ramawh}
A.~Ram.
\newblock Alcove walks, {H}ecke algebras, spherical functions, crystals and
  column strict tableaux.
\newblock {\em Pure Appl. Math. Q.}, 2:963--1013, 2006.

\bibitem{raycfm}
A.~Ram and M.~Yip.
\newblock A combinatorial formula for {M}acdonald polynomials.
\newblock {\tt arXiv:math/0803.1146}.

\bibitem{schghl}
C.~Schwer.
\newblock Galleries, {H}all-{L}ittlewood polynomials, and structure constants
  of the spherical {H}ecke algebra.
\newblock {\em Int. Math. Res. Not.}, pages Art. ID 75395, 31, 2006.

\end{thebibliography}

\end{document}